\documentclass[11pt,a4paper]{article}
\usepackage{amsfonts}
\usepackage{amsmath}
\usepackage{amssymb}
\usepackage{amstext}
\usepackage{amsbsy,upref}
\usepackage{graphicx}
\usepackage{color}

\newtheorem{theorem}{Theorem}
\newtheorem{lemma}[theorem]{Lemma}
\newtheorem{proposition}[theorem]{Proposition}
\newtheorem{corollary}[theorem]{Corollary}

\newtheorem{assumption}[theorem]{Assumption}
\newtheorem{remark}{Remark}[theorem]
\newenvironment{proof}[1][Proof]{\noindent\textbf{#1.} }{\mbox{}\hfill $\Box$\smallskip}

\begin{document}

\date{\vspace{-12.5mm}}
\title{A sufficient condition for a discrete spectrum of the Kirchhoff plate
with an infinite peak\thanks{%
The work is supported by the Russian Foundation for Basic Research (project
12-01-00348). The first author is also supported by the Chebyshev Laboratory
(Department of Mathematics and Mechanics, Saint-Petersburg State University)
under the grant 11.G34.31.0026 of the Government of the Russian Federation
and grant 6.38.64.2012 of Saint-Petersburg State University.}}
\author{F.L. Bakharev, S.A. Nazarov, G.H. Sweers}
\maketitle

\begin{abstract}
Sufficient conditions for a discrete spectrum of the biharmonic equation in
a two-dimensional peak-shaped domain are established. Different boundary
conditions from Kirchhoff's plate theory are imposed on the boundary and the
results depend both on the type of boundary conditions and the sharpness
exponent of the peak.
\end{abstract}


\section{Motivation}
Elliptic boundary value problems on domains which have a Lipschitz boundary
and a compact closure, in particular when they generate positive self-adjoint
operators, have fully discrete spectra. However, if the domain loses the
Lipschitz property or compactness, other situations may occur. It is well-known
that for the Dirichlet case boundedness is sufficient but not necessary for
a discrete spectrum. See the famous paper by Rellich \cite{rellich} or the more
recent contributions by \cite{rosenb,vdB}. On the other hand, for the Neumann problem
of the Laplace operator there exist numerous examples of bounded domains such that
the spectrum gets a non-empty continuous component (see e.g.
\cite{KuGi, MaPo, MaPo-e, Simon, Hempel-ea}).

The literature on the spectra for the
Laplace operator with various boundary conditions is focussed on domains that
have a cusp, a finite or infinite peak or horn
\cite{vdB, Hempel-ea, JakMolS, DavSim, Jak, Ivrii, BoLe, vdBL, Kov} or even considered
a rolled horn \cite{Simon}.

The criteria in \cite{AdFu} and \cite{Ev} for the embedding $H^{1}(\Omega )\subset
L_{2}(\Omega )$ to be compact, show that the Neumann-Laplace problem on a
domain $\Omega $ with the infinite peak
\begin{equation}
\Pi _{R}=\{x=(x_{1},x_{2})\in {\mathbb{R}^{2}}:\quad
x_{1}>R,-H(x_{1})<x_{2}<H(x_{1})\},  \label{0.1}
\end{equation}%
where the function $H>0$ is smooth and monotone decreasing, has a discrete spectrum if and
only if
\begin{equation}
\lim\limits_{y\rightarrow +\infty }\int_{y}^{+\infty }\frac{H(\eta )}{H(y)}%
\,d\eta =0\quad (\Leftrightarrow \quad \lim\limits_{y\rightarrow +\infty }%
\frac{H(y+\epsilon )}{H(y)}=0\,\,\mbox{\rm for\,any}\,\,\epsilon >0).
\label{0.2}
\end{equation}%
The function $H$ is assumed to have a first derivative that tends to zero
and a bounded second derivative. It will be convenient to use the notation $%
\Upsilon (y)=(-H(y),H(y))$. 
Here is an image of such a domain:\nopagebreak

\noindent\resizebox{\textwidth}{!}{\includegraphics{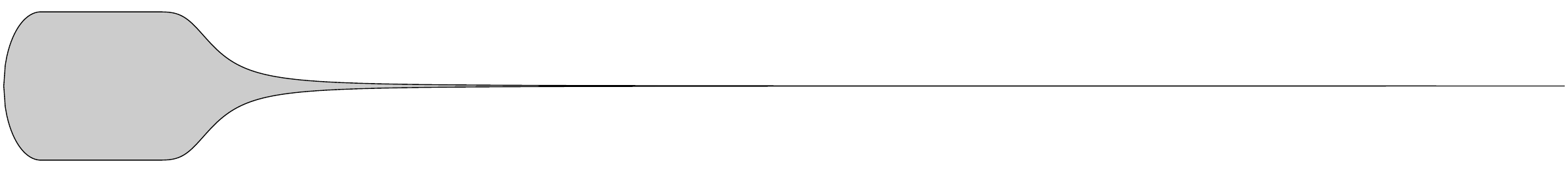}}

The simplest boundary irregularity violating the Lipschitz condition, is but
the (finite) peak
\begin{equation}
\varpi _{R}=\{x:0<x_{1}<R,-h(x_{1})<x_{2}<h(x_{1})\},  \label{0.3}
\end{equation}%
where $h(x_{1})=h_{0}x_{1}^{1+\alpha }$, $h_{0}>0$ and $\alpha >0$.
Nevertheless, the spectrum of the Neumann problem in the domain with this
peak stays discrete. See Remark \ref{bdd}.

A criterion~(\cite{na424}) for an essential spectrum in
the Neumann problem for elliptic systems of second order differential
equations with a polynomial property is derived in \cite{na262}. In
particular it shows that the continuous spectrum of an elastic body with $%
\alpha \geq 1$ for the peak \eqref{0.3}, is non-empty (see~\cite{na407,
na431}). This phenomenon of generating wave processes in a finite volume, is
known experimentally and used in the engineering practice to construct wave
dampers, \textquotedblleft black holes\textquotedblright , for elastic
oscillations (see~\cite{Mir1,Kryl}, etc.).

In this paper we study the spectra of boundary value problems for the
Kirchhoff model of a thin elastic plate described by the biharmonic operator 
$\Delta^{2}$. We consider the three mechanically most reasonable cases, namely
where the lateral sides of the peak are supplied with one of the following
three types of the boundary conditions: the clamped edge 
(\textit{\textbf{D}irichlet}), the traction-free edge (\textit{\textbf{N}eumann}) 
and the
hinged edge (\textit{\textbf{M}ixed}). In all the cases the spectrum of the
problem in a bounded domain with the peak \eqref{0.3} is discrete. We derive
sufficient conditions for the spectrum to be discrete for the boundary value
problem on an unbounded domain with a peak as in \eqref{0.1}.

If a \textquotedblleft sufficient number\textquotedblright\ of Dirichlet
conditions are imposed on the lateral sides of the peak (the cases \textit{%
\textbf{D--D}}, \textit{\textbf{D--M}}, \textit{\textbf{D--N}} and \textit{%
\textbf{M--M}}; see formulas (\ref{0.5}--\ref{0.7}) and~\eqref{0.13},~%
\eqref{0.14}), then the proof that the spectrum is discrete becomes rather
simple (Theorem \ref{Theorem1}). Indeed, it suffices to apply the weighted
Friedrich's inequality \eqref{0.15} and to take into account the decay of
the quantity $H(y)$ as $y\rightarrow +\infty $. With a different argument,
it is straightforward to obtain a condition for the discrete spectrum if the
peak's edges are traction-free (the case \textit{\textbf{N--N}}; see
Corollary \ref{Corollary2}). Indeed, as shown in~\cite{AdFu}, the second
condition in \eqref{0.2} implies a criterion for the compact embedding ${H}%
^{m}(\Omega )\subset L^{2}(\Omega )$ for all $m$ (we just need $m=2$).
Therefore, we are only interested into investigating the case \textit{%
\textbf{M--N}}, i.e., one side of the peak is traction-free and the other
hinged. We obtain a sufficient condition for the discrete spectrum (Theorem %
\ref{Theorem3}) by applying weighted inequalities of Hardy type (Lemmas \ref%
{Lemma1} and \ref{Lemma2}). This approach differs from the one used in \cite%
{AdFu,Ev}, and we can also use it for the case \textit{\textbf{N--N}}
(Theorem \ref{Theorem2}).

The obtained results essentially differ from each other: under the
conditions \eqref{0.13} and also under \eqref{0.14} any decay of $H$ is
enough, but the case \textit{\textbf{M--N}} needs a power decay rate with
the exponent $\alpha >1$ while the case \textit{\textbf{N--N}} needs a
superexponential decay rate (see comments to Theorems \ref{Theorem2} and \ref%
{Theorem3}).

\section{The Kirchhoff plate model}
Let $\Omega $ be a domain in the plane ${\mathbb{R}}^{2}$ with a smooth (of
class $C^{\infty }$) boundary $\partial \Omega $, coinciding with the peak %
\eqref{0.1} inside the half-plane $\{x:x_{1}>R\}$ and being bounded outside
this half-plane. We regard $\Omega $ as the projection of a thin isotropic
homogeneous plate and apply the Kirchhoff theory (see \cite[\S 30]{Mikh},
\cite[Ch.\,7]{Nabook} etc). So we arrive at the fourth-order differential
equation
\begin{equation}
\Delta ^{2}u(x)=\lambda u(x),\quad x\in \Omega ,  \label{0.4}
\end{equation}%
which describes transverse oscillations of the plate. Here, $u(x)$ is the
plate deflection, and $\lambda $ a spectral parameter proportional to the
square of the oscillation frequency.

The following sets of boundary conditions have a clear physical
interpretation (see \cite[\S 30]{Mikh}, \cite[\S 1.1]{Guido} etc):

\begin{description}
\item[(\textit{D})] Dirichlet for a clamped edge:
\begin{equation}
u(x)=\partial _{n}u(x)=0,\quad x\in\Gamma _{D}.  \label{0.5}
\end{equation}

\item[(\textit{N})] Neumann for a traction-free edge:
\begin{equation}
\left\{
\begin{array}{c}
\partial _{n}\Delta u(x)-(1-\nu )(\partial _{s}\varkappa(x) \partial
_{s}u(x)-\partial _{s}^{2}\partial _{n}u(x))=0,\smallskip \\
\Delta u(x)-(1-\nu )(\partial _{s}^{2}u(x)+\varkappa(x) \partial _{n}u(x))=0,%
\end{array}%
\right. x\in \Gamma _{N}.  \label{0.6}
\end{equation}

\item[(\textit{M})] Mixed for a hinged edge:
\begin{equation}
u=\Delta u(x)- (1-\nu )\varkappa(x)\partial _{n}u(x)=0,\quad x\in\Gamma _{M}.
\label{0.7}
\end{equation}
\end{description}

Here, $\partial _{n}$ and $\partial _{s}$ stand for the normal and
tangential derivatives, $\varkappa (x)$ is the signed curvature of the
contour $\Gamma $ at the point $x\in \Gamma $ positive for convex boundary
parts, and $\nu \in \lbrack 0,1/2)$ is the Poisson ratio. Finally, $\Gamma
_{D}$, $\Gamma _{N}$, and $\Gamma _{M}$ are the unions of finite families of
open curves and $\Gamma =\overline{\Gamma _{D}}\cup \overline{\Gamma _{N}}%
\cup \overline{\Gamma _{M}}$, two of which may be empty. In what follows it
is convenient to use the notation $y=x_{1}$ and $z=x_{2}$.

The general properties of the spectra depend on which of the boundary
conditions (\ref{0.5}--\ref{0.7}) are imposed on the upper ($+$) and lower ($%
-$) sides $\Sigma _{R}^{\pm }=\{x:y>R,z=\pm H(y)\}$ of the peak. To be more
precise, we consider the spectral problem of the variational formulation
\begin{equation}
a(u,v)=\lambda (u,v)_{\Omega },\quad v\in {\mathbb{H}}\,,  \label{0.9}
\end{equation}%
where $(\cdot,\cdot)_{\Omega }$ is a scalar product in the Lebesgue space $%
L_{2}(\Omega )$, $a$ is the symmetric bilinear form such that $\frac{1}{2}%
a(u,u)$ is the elastic energy stored in the plate:
\begin{equation}
a\left( u,u\right) =\int_{\Omega }\left( \left\vert \frac{\partial ^{2}u}{%
\partial x_{1}^{2}}\right\vert ^{2}+\left\vert \frac{\partial ^{2}u}{%
\partial x_{2}^{2}}\right\vert ^{2}+2\left( 1-\nu \right) \left\vert \frac{%
\partial ^{2}u}{\partial x_{1}\partial x_{2}}\right\vert ^{2}+2\nu \frac{%
\partial ^{2}u}{\partial x_{1}^{2}}\frac{\partial ^{2}u}{\partial x_{2}^{2}}%
\right) dx,  \label{0.10}
\end{equation}%
and ${\mathbb{H}}$ is the subspace of functions $u\in H^{2}(\Omega )$,
satisfying the conditions \eqref{0.5} on $\Gamma _{D}$ and $u=0$ on $\Gamma
_{M}$ in the sense of traces. One directly verifies that
\begin{equation*}
a(u,u)\geq (1-\nu )\sum_{j,k=1}^{2}\int_{\Omega }\left\vert \frac{\partial
^{2}u}{\partial x_{j}\partial x_{k}}\right\vert ^{2}\,dx.
\end{equation*}%
Hence, the bilinear form on the left-hand side of \eqref{0.9} is positive
and closed on ${\mathbb{H}}\times {\mathbb{H}}$, and according to \cite[\S %
10.1]{BiSo} we can associate a unbounded positive self-adjoint operator $A$
in the Hilbert space $L_{2}(\Omega )$ to problem \eqref{0.9} with domain $%
D(A)={\mathbb{H}}$. Its spectrum $\sigma $ belongs to $\left[ 0,\infty
\right) $ and, moreover, the spectrum is discrete if and only if the
embedding ${\mathbb{H}}\subset L_{2}(\Omega )$ is compact (cf. \cite[Theorem
10.1.5]{BiSo}). This is what we call the spectrum for (\ref{0.4}--\ref{0.7})
and which will be studied in the present paper. The remaining boundary
conditions in (\ref{0.5}--\ref{0.7}) appear as intrinsic natural conditions
from (\ref{0.9}--\ref{0.10}), see again \cite[\S 30]{Mikh}, \cite[\S 1.1]%
{Guido} etc.

\section{Simple cases}
Assume that the chosen boundary conditions provide at least one of the
following two groups of relations
\begin{eqnarray}
i) & u=0\hspace{1cm}\text{ on }\Sigma _{R}^{+}\cup \Sigma _{R}^{-},\hspace{%
10mm}  \label{0.13} \\
ii) & u=\partial _{n}u= 0\ \text{ on }\Sigma _{R}^{+}\quad\text{ (or on }%
\Sigma _{R}^{-}\text{)}.  \label{0.14}
\end{eqnarray}%
In both the cases \eqref{0.13} and \eqref{0.14} the following version of
Friedrich's inequality is valid
\begin{equation}
\int_{\Upsilon (y)}\left\vert \frac{\partial ^{2}u}{\partial z^{2}}%
(y,z)\right\vert ^{2}\,dz\geq \frac{c}{H(y)^{4}}\int_{\Upsilon
(y)}\left\vert u(y,z)\right\vert ^{2}\,dz\,,  \label{0.15}
\end{equation}%
and, therefore,
\begin{equation}
a(u,u)\geq c\int_{\Omega }H(y)^{-4}\left\vert u(x)\right\vert ^{2}\,dx\,.
\label{0.16}
\end{equation}%
The embedding operator $\gamma :{\mathbb{H}}\rightarrow L^{2}(\Omega )$ can
be represented as the sum $\gamma _{0}+\gamma _{\rho}$, where $\rho\geq R$
is large positive, $\gamma _{0}=\gamma -\gamma _{\rho}$, and $\gamma _{\rho}$
includes the operator of multiplication by the characteristic function of $%
\Pi _{\rho}$. The operator $\gamma _{0}$ is compact, and the norm of $\gamma
_{\rho}$, in view of \eqref{0.16}, does not exceed $c\max \left\{
H(x_{1})^{-2};~x_{1}\geq \rho\right\} $. Since the function $H$ decays, this
quantity goes to zero when $\rho\rightarrow +\infty $, i.e., the operator $%
\gamma $ can be approximated by compact operators in the operator norm.
Thus, $\gamma $ is compact and the following result is proved.

\begin{theorem}
\label{Theorem1} Suppose that the boundary conditions for problem 
\eqref{0.4} as given in {\upshape (\ref{0.5}-\ref{0.7})} contain one of the cases \eqref{0.13}
or \eqref{0.14}. Then the spectrum is discrete.
\end{theorem}

\begin{remark}\label{bdd} By a similar reasoning, we may conclude that in the bounded domain $\omega $, 
with the peak as in \eqref{0.3}, equation \eqref{0.4} has a discrete
spectrum for any set of conditions (\ref{0.5}--\ref{0.7}) on the arc~$%
\partial \omega \setminus \mathcal{O}$. This fact follows from the
inequality (see \cite{na401}):
\begin{equation*}
\left\Vert \left\vert x\right\vert ^{-1}u;L_{2}(\omega )\right\Vert ^{2}\leq
c\left( \left\Vert \nabla u;L_{2}(\omega )\right\Vert ^{2}+\left\Vert
u;L_{2}(\omega \setminus \varpi _{R})\right\Vert ^{2}\right) \,.
\end{equation*}
\end{remark}

\section{Auxiliary inequalities}
First of all we prove some one-dimensional weighted inequalities, two of
which are of Hardy type involving a weight function $h$ as follows.

\begin{assumption}
\label{ass2}Let $h$ be a positive weight function of class $C^{2}$ on $%
[0,+\infty )$ satisfying 
\begin{itemize}
\item $\displaystyle\int_{0}^{\infty }h\left( s\right)
ds<\infty \ $ and 
\item for some $T$ large $h^{\prime }\left( t\right) <0$ and $%
h^{\prime \prime }\left( t\right) >0$ for $t\in \left( T,\infty \right) $.
\end{itemize}
\end{assumption}

Throughout this section $h$ is supposed to satisfy this assumption.

\begin{lemma}
\label{Lemma1}If $U$ is differentiable for $y\geq R$ and $U(R)=0$, then
\begin{equation*}
\int_{R}^{+\infty }h(y)\left\vert U(y)\right\vert ^{2}\,dy\leq
\int_{R}^{+\infty }F_{h}(y)\left\vert \partial _{y}U(y)\right\vert ^{2}\,dy,
\end{equation*}%
where $F_{h}(y)=\frac{4}{h(y)}\left( \int_{y}^{+\infty }h(\tau )\,d\tau
\right) ^{2}$.
\end{lemma}

\begin{proof}
Using the Cauchy-Bunyakovsky-Schwarz inequality, we have
\begin{gather*}
\int_{R}^{+\infty }h(y)\left\vert U(y)\right\vert
^{2}\,dy=2\int_{R}^{+\infty }h(y)\int_{R}^{y}\partial _{t}U(t)U(t)dt\,dy\leq
\\
\leq 2\int_{R}^{+\infty }\int_{t}^{+\infty }h(y)\left\vert \partial
_{t}U(t)U(t)\right\vert \,dy\,dt\leq \\
2\left( \int_{R}^{+\infty }h(t)|U(t)|^{2}\,dt\right) ^{1/2}\left(
\int_{R}^{+\infty }h(t)^{-1}\left( \int_{t}^{+\infty }h(y)\,dy\right)
^{2}|\partial _{t}U(t)|^{2}\,dt\right) ^{1/2},
\end{gather*}%
and the result follows through division by a common factor.
\end{proof}

\begin{lemma}
\label{Lemma2}If $U$ is differentiable for $y\geq R$ and $U(R)=0$, then
\begin{equation*}
\int_{R}^{+\infty }h(y)\left\vert \partial _{y}U(y)\right\vert ^{2}\,dy\geq
\int_{R}^{+\infty }G_{h}(y)\left\vert U(y)\right\vert ^{2}\,dy,
\end{equation*}%
where $G_{h,R}(y)=\frac{1}{4h(y)}\left( \int_{R}^{y}h(\tau )^{-1}\,d\tau
\right) ^{-2}$.
\end{lemma}

\begin{proof}
For functions $v$ with $v\left( 0\right) =0$ a Hardy inequality tells us
that
\begin{equation*}
\int_{0}^{+\infty }t^{-2}\left\vert v(t)\right\vert ^{2}\,dt\leq
4\int_{0}^{+\infty }\left\vert \partial _{t}v(t)\right\vert ^{2}\,dt.
\end{equation*}%
We make the change $t\mapsto y\in \lbrack R,+\infty )$ where $%
t=\int_{R}^{y}h(\tau )^{-1}\,d\tau $, and set $U(y)=v(t)$. Then $\partial
_{t}v(t)=h(y)\partial _{y}U(y)$ leads to the desired estimate.
\end{proof}

\begin{corollary}
\label{Corollary1}If the function $U$ is twice differentiable for $y\geq R$
and $U(R)=U^{\prime }(R)=0$, then
\begin{equation}
\int_{R}^{+\infty }h(t)\left\vert \partial _{t}^{2}U(t)\right\vert
^{2}\,dt\geq W_{h}(R)\int_{R}^{+\infty }h(t)\left\vert U(t)\right\vert
^{2}\,dt  \label{one-dim-ineq}
\end{equation}%
is valid with
\begin{equation}
W_{h}(R):=\inf_{t\in \lbrack R,+\infty )}\frac{G_{h,R}(t)}{F_{h}(t)}%
=\inf_{t\in \lbrack R,+\infty )}\frac{1}{16}\left( \int_{R}^{t}h(\tau
)^{-1}\,d\tau \right) ^{\hspace{-1mm}-2}\hspace{-1mm}\left( \int_{t}^{\infty
}h(\tau )\,d\tau \right) ^{\hspace{-1mm}-2}\hspace{-3mm}.  \label{NW}
\end{equation}
\end{corollary}

The following lemma will show that the assumptions in Theorem \ref{Theorem2}
can be met.

\begin{lemma}
\label{tech3}We have%
\begin{equation}
\inf_{t\in \lbrack R,+\infty )}\frac{G_{h^{3},R}(t)}{h\left( t\right)
h^{\prime }(t)^{2}}\geq 1.  \label{logres}
\end{equation}%
Suppose moreover that%
\begin{equation}
\lim_{t\rightarrow +\infty }\partial _{t}\left( \log \left( h\left( t\right)
\right) \right) =-\infty .  \label{log}
\end{equation}%
Then
\begin{equation*}
W_{h}(R)\rightarrow +\infty \text{ and }W_{h^{3}}(R)\rightarrow +\infty
\text{ for }R\rightarrow \infty .
\end{equation*}
\end{lemma}

\begin{proof}
Since $-h^{\prime }\left( \tau \right) \geq -h^{\prime }\left( t\right) >0$
for $\tau <t$, we find
\begin{gather*}
\frac{G_{h^{3},R}(t)}{h\left( t\right) h^{\prime }(t)^{2}}=\frac{1}{%
4h^{\prime }(t)^{2}h(t)^{4}}\left( \int_{R}^{t}h(\tau )^{-3}\,d\tau \right)
^{-2}\geq \\
\geq \frac{1}{4h(t)^{4}}\left( -\int_{R}^{t}h^{\prime }\left( \tau \right)
h(\tau )^{-3}\,d\tau \right) ^{-2}=\frac{1}{4h(t)^{4}}\left( \frac{1}{%
2h\left( t\right) ^{2}}-\frac{1}{2h\left( R\right) ^{2}}\right) ^{-2}\geq 1.
\end{gather*}%
Since (\ref{log}) equals $h^{\prime }\left( t\right) /h\left( t\right)
\rightarrow -\infty $ for $t\rightarrow \infty $, we find that for $%
t\rightarrow \infty $ both%
\begin{equation*}
h\left( t\right) \left( \int_{t}^{\infty }h(\tau )\,d\tau \right)
^{-1}\rightarrow \infty \text{ and }\frac{1}{h\left( t\right) }\left(
\int_{R}^{t}\frac{1}{h(\tau )}d\tau \right) ^{-1}\rightarrow \infty .
\end{equation*}%
Hence $G_{h,R}(t)/F_{h}(t)\rightarrow \infty $ for $t\rightarrow \infty $
and since the quotient also goes to infinity for $t\downarrow R$, it has a
minimum in some $t_{R}\in \left( R,\infty \right) $. Calculating $\left(
G_{h,R}(t)/F_{h}(t)\right) ^{\prime }=0$ we find
\begin{equation*}
\frac{1}{h\left( t\right) }\int_{t}^{\infty }h(\tau )\,d\tau -h\left(
t\right) \int_{R}^{t}h(\tau )^{-1}\,d\tau =0.
\end{equation*}%
Hence
\begin{equation*}
\inf_{t\in \left[ R,\infty \right) }\left( \int_{R}^{t}h(\tau )^{-1}\,d\tau
\ \int_{t}^{\infty }h(\tau )\,d\tau \right) ^{-1}=h\left( t_{R}\right)
^{2}\left( \int_{t_{R}}^{\infty }h(\tau )\,d\tau \right) ^{-2},
\end{equation*}%
which goes to infinity for $R\rightarrow \infty $\ since $t_{R}>R$. The
claim for $W_{h}(R)$ follows. The same derivation holds true for $%
W_{h^{3}}(R)$.
\end{proof}

\section{The traction-free boundary and a 2d-estimate}
We assume that the Neumann boundary conditions \eqref{0.6} are imposed at
the both sides of the peak \eqref{0.1}. While $\rho\to +\infty$ let us
describe the behavior of multiplier $K(\rho)$ in the inequality
\begin{equation}
K(\rho)\int_{\Pi _{\rho}}|u(y,z)|^{2}\,dy\leq \Vert u;{H}^{2}(\Omega )\Vert
^{2}\,,\quad u\in H^{2}(\Omega ).  \label{K}
\end{equation}%
If $K(\rho)$ increases unboundedly as $\rho\rightarrow +\infty $ then, as
above, Theorem 10.1.5 \cite{BiSo} ensures that the spectrum of the problem (%
\ref{0.4}--\ref{0.7}) stays discrete even in the case both sides of the peak
are supplied with the traction-free boundary conditions (N) and, moreover,
for any other boundary conditions from the list (\ref{0.5}--\ref{0.7}).

\begin{proposition}
Suppose that Assumption \ref{ass2} is satisfied and that 
\begin{equation*}
\lim\limits_{t\rightarrow \infty }\partial _{t}\left( \log \left( H\left(
t\right) \right) \right) =-\infty.
\end{equation*} 
Then for $\rho $ sufficiently large \eqref{K} holds true with%
\begin{equation*}
K\left( \rho \right) =c\min \{H^{-4}(\rho ),W_{H}(\rho ),W_{H^{3}}(\rho )\}.
\end{equation*}
\end{proposition}

\begin{proof}
It is sufficient to check the inequality \eqref{K} for smooth functions
which vanish for $y<\rho $. We use the representation
\begin{equation*}
u(x)=u(y,z)=u_{0}(y)+zu_{1}(y)+u^{\perp }(y,z)
\end{equation*}%
where, for $y>\rho $, the component $u^{\perp }$ is subject to the following
two orthogonality conditions
\begin{equation}
\int_{\Upsilon (y)}\hspace*{-0.4cm}u^{\perp }(y,z)\,dz=0\text{ and }%
\int_{\Upsilon (y)}\hspace*{-0.4cm}\partial _{z}u^{\perp }(y,z)\,dz=u^{\perp
}(y,H(y))-u^{\perp }(y,-H(y))=0.  \label{ortog}
\end{equation}

Let us process the integrals on the right-hand side of
\begin{gather}
\int_{\Pi _{\rho }}\left\vert \nabla _{x}^{2}u(x)\right\vert
^{2}\,dx=I_{1}+4I_{2}+I_{3},\text{ where}  \label{star} \\
\begin{array}{ccc}
I_{1}:=\displaystyle\int_{\Pi _{\rho }}\left\vert \partial
_{z}^{2}u(x)\right\vert ^{2}dx, & I_{2}:=\displaystyle\int_{\Pi _{\rho
}}\left\vert \partial _{y}\partial _{z}u(x)\right\vert ^{2}dx, & I_{3}:=%
\displaystyle\int_{\Pi _{\rho }}\left\vert \partial _{y}^{2}u(x)\right\vert
^{2}dx.%
\end{array}
\notag
\end{gather}%
Since $I_{1}=\int_{\rho }^{+\infty }\int_{\Upsilon (y)}\left\vert \partial
_{z}^{2}u^{\perp }(y,z)\right\vert ^{2}\,dz\,dy$, since by the orthogonality
conditions in \eqref{ortog} also here inequality \eqref{0.15} holds, we find
\begin{equation}
I_{1}\geq c\int_{\Pi _{\rho }}H(y)^{-4}\left\vert u^{\perp }(x)\right\vert
^{2}\,dx.  \label{I1est}
\end{equation}%
For the last term in \eqref{star} we have
\begin{gather}
I_{3}=\int_{\Pi _{\rho }}\left\vert \partial _{y}^{2}u_{0}(y)+z\partial
_{y}^{2}u_{1}(y)+\partial _{y}^{2}u^{\perp }(y,z)\right\vert ^{2}\,dx\geq
J_{1}+J_{2}+2J_{3}+2J_{4},  \label{star3} \\
\begin{array}{cll}
\text{where} & J_{1}=\displaystyle\int_{\Pi _{\rho }}\left\vert \partial
_{y}^{2}u_{0}(y)\right\vert ^{2}dx, & J_{3}=\displaystyle\int_{\Pi _{\rho
}}\partial _{y}^{2}u_{0}(y)\partial _{y}^{2}u^{\perp }(y,z)dx, \\
& J_{2}=\displaystyle\int_{\Pi _{\rho }}\left\vert z\partial
_{y}^{2}u_{1}(y)\right\vert ^{2}dx, & J_{4}=\displaystyle\int_{\Pi _{\rho
}}z\partial _{y}^{2}u_{1}(y)\partial _{y}^{2}u^{\perp }(y,z)dx.%
\end{array}
\notag
\end{gather}%
We readily notice that according to the inequality \eqref{one-dim-ineq} the
estimates%
\begin{equation}
J_{1}\geq W_{H}(\rho )\int_{\rho }^{+\infty }2H(y)\left\vert
u_{0}(y)\right\vert ^{2}\,dy=W_{H}(\rho )\int_{\Pi _{\rho }}\left\vert
u_{0}(y)\right\vert ^{2}\,dx,  \label{J1est}
\end{equation}%
\begin{gather}
J_{2}=\tfrac{2}{3}\int_{\rho }^{+\infty }H^{3}(y)\left\vert \partial
_{y}^{2}u_{1}(y)\right\vert ^{2}\,dy\ \geq  \notag \\
\geq \tfrac{2}{3}\,W_{H^{3}}(\rho )\int_{\rho }^{+\infty }H^{3}(y)\left\vert
u_{1}(y)\right\vert ^{2}\,dy=W_{H^{3}}(\rho )\int_{\Pi _{\rho }}\left\vert
zu_{1}(y)\right\vert ^{2}\,dx  \label{J2est}
\end{gather}%
are fulfilled. For our purpose we need $W_{H}(\rho )\rightarrow +\infty $
and $W_{H^{3}}\left( \rho \right) \rightarrow +\infty $ for $\rho
\rightarrow \infty $ and this we will assume.

Besides, by the Cauchy-Bunyakovsky-Schwarz inequality, we have
\begin{equation*}
\left\vert J_{3}\right\vert \leq J_{1}^{1/2}\left( \int_{\rho }^{+\infty }%
\frac{1}{2H(y)}\left( \int_{\Upsilon (y)}\partial _{y}^{2}u^{\perp
}(y,z)\,dz\right) ^{2}\,dy\right) ^{1/2}\,.
\end{equation*}%
We now deal with the inner integral in $z$ in the last expression. To this
end, we take into account the orthogonality conditions \eqref{ortog} and the
trace inequality. We differentiate the first equality in \eqref{ortog} twice
with respect to $y$ and obtain
\begin{multline*}
\sum_{\pm }\Bigl(2\partial _{y}u^{\perp }(y,\pm H(y))\partial
_{y}H(y)+u^{\perp }(y,\pm H(y))\partial _{y}^{2}H(y)\ \pm \\
\pm \partial _{z}u^{\perp }(y,\pm H(y))\left( \partial _{y}H(y)\right) ^{2}%
\Bigr)+\int_{\Upsilon (y)}\partial _{y}^{2}u^{\perp }(y,z)\,dz=0.
\end{multline*}%
Thus,
\begin{multline*}
\left( \int_{\Upsilon (y)}\partial _{y}^{2}u^{\perp }(y,z)\,dz\right)
^{2}\leq c\sum_{\pm }\bigg(|\partial _{z}u^{\perp }(y,\pm
H(y))|^{2}|\partial _{y}H(y)|^{4}\ + \\
+\ |u^{\perp }(y,\pm H(y))|^{2}|\partial _{y}^{2}H(y)|^{2}+|\partial
_{y}u^{\perp }(y,\pm H(y))|^{2}|\partial _{y}H(y)|^{2}\bigg)\,.
\end{multline*}%
For the first two terms between the brackets we use the trace inequality
\begin{equation*}
\left\vert \partial _{z}u^{\perp }(y,\pm H(y))\right\vert
^{2}|H(y)|^{2}+\left\vert u^{\perp }(y,\pm H(y))\right\vert ^{2}\leq
c\left\vert H(y)\right\vert ^{3}\int_{\Upsilon (y)}\left\vert \partial
_{z}^{2}u^{\perp }(y,z)\right\vert ^{2}\,dz\,.
\end{equation*}%
For the third term, we write down the chain of inequalities
\begin{multline*}
\left\vert \partial _{y}u^{\perp }(y,\pm H(y))\right\vert ^{2}\leq
cH(y)\int_{\Upsilon (y)}\left\vert \partial _{y}\partial _{z}u^{\perp
}(y,z)\right\vert ^{2}\,dz\ + \\
+\ c|H(y)|^{-2}\left( \int_{\Upsilon (y)}\partial _{y}u^{\perp
}(y,z)\,dz\right) ^{2}\leq cH(y)\int_{\Upsilon (y)}\left\vert \partial
_{y}\partial _{z}u^{\perp }(y,z)\right\vert ^{2}\,dz\ + \\
+\ 2c\left( \frac{\partial _{y}H(y)}{H(y)}\right) ^{2}\left( \left\vert
u^{\perp }(y,H(y))\right\vert ^{2}+\left\vert u^{\perp }(y,-H(y))\right\vert
^{2}\right) \,.
\end{multline*}%
As a result, we find that
\begin{multline*}
\left\vert \int_{\Upsilon (y)}\partial _{y}^{2}u^{\perp
}(y,z)\,dz\right\vert ^{2}\leq c\left( |\partial _{y}H(y)|^{4}H(y)+|\partial
_{y}^{2}H(y)|^{2}|H(y)|^{3}\right) \times \\
\times \int_{\Upsilon (y)}\left\vert \partial _{z}^{2}u^{\perp
}(y,z)\right\vert ^{2}\,dz+c\left\vert \partial _{y}H(y)\right\vert
^{2}\left\vert H(y)\right\vert \int_{\Upsilon (y)}\left\vert \partial
_{y}\partial _{z}u^{\perp }(y,z)\right\vert ^{2}\,dz\,.
\end{multline*}

The final inequality for the integral $J_{3}$ takes the form%
\begin{equation}
\left\vert J_{3}\right\vert \leq c_{1}\left( \rho \right)
J_{1}^{1/2}I_{1}^{1/2}+c_{2}\left( \rho \right) J_{1}^{1/2}K_{1}^{1/2}
\label{J3est}
\end{equation}
where $K_{1}=\left\Vert \partial _{yz}^{2}u^{\perp };L_{2}(\Pi _{\rho
})\right\Vert ^{2}$ and%
\begin{eqnarray*}
c_{1}\left( \rho \right) &=&c\sup_{y\in \lbrack \rho ,+\infty )}\left(
\left\vert \partial _{y}H(y)\right\vert ^{2}+\left\vert \partial
_{y}^{2}H(y)\right\vert \left\vert H(y)\right\vert \right) , \\
c_{2}\left( \rho \right) &=&c\sup_{y\in \lbrack \rho ,+\infty )}\left\vert
\partial _{y}H(y)\right\vert .
\end{eqnarray*}
Note that both supremums tend to $0$ for $\rho \rightarrow +\infty $. A
similar argument shows that%
\begin{equation}
\left\vert J_{4}\right\vert \leq c_{1}\left( \rho \right)
J_{2}^{1/2}I_{1}^{1/2}+c_{2}\left( \rho \right) J_{2}^{1/2}K_{1}^{1/2}.
\label{J4est}
\end{equation}

It remains to process the second term in \eqref{star}, that is,
\begin{gather*}
I_{2}=\int_{\Pi _{\rho }}\left\vert z\partial _{y}u_{1}(y)+\partial
_{y}\partial _{z}u^{\perp }(y,z)\right\vert ^{2}\,dx= \\
\int_{\Pi _{\rho }}\left\vert \partial _{y}\partial _{z}u^{\perp
}(y,z)\right\vert ^{2}\,dx\ +\int_{\Pi _{\rho }}\left\vert \partial
_{y}u_{1}(y)\right\vert ^{2}\,dx+2\int_{\Pi _{\rho }}\partial
_{y}u_{1}(y)\partial _{y}\partial _{z}u^{\perp }(y,z)\,dx.
\end{gather*}%
\vspace{-7mm}
\begin{equation*}
\hspace{2cm}\rotatebox{315}{$=$}\raisebox{-5mm}{$K_1$}\hspace{33mm}%
\rotatebox{315}{$=:$}\raisebox{-5mm}{$K_2$}\hspace{33mm}\rotatebox{315}{$=:$}%
\raisebox{-5mm}{$2 K_3$}
\end{equation*}%
So it follows that
\begin{equation}
K_{1}=I_{2}-K_{2}-2K_{3}\leq I_{2}+2\left\vert K_{3}\right\vert .
\label{K1byK3}
\end{equation}%
We continue by estimating the integral $K_{3}$:
\begin{eqnarray*}
&&\left\vert K_{3}\right\vert =\left\vert \int_{\rho }^{+\infty }\partial
_{y}u_{1}(y)\int_{\Upsilon (y)}\partial _{y}\partial _{z}u^{\perp
}(y,z)\,dz\,dy\right\vert \\
&\leq &\left( \int\limits_{\rho }^{+\infty }G_{H^{3},\rho }(y)|\partial
_{y}u_{1}(y)|^{2}\,dy\right) ^{\frac{1}{2}}\left( \int\limits_{\rho
}^{+\infty }G_{H^{3},\rho }(y)^{-1}\left\vert \int_{\Upsilon (y)}\partial
_{y}\partial _{z}u^{\perp }(y,z)\,dz\right\vert ^{2}\,dy\right) ^{\frac{1}{2}%
} \\
&\leq &cJ_{2}^{1/2}\left( \int_{\rho }^{+\infty }G_{H^{3},\rho
}(y)^{-1}\left\vert \int_{\Upsilon (y)}\partial _{y}\partial _{z}u^{\perp
}(y,z)\,dz\right\vert ^{2}\,dy\right) ^{1/2}.
\end{eqnarray*}%
Differentiating the second formula \eqref{ortog} with respect to $y$ yields
\begin{equation*}
\int_{\Upsilon (y)}\partial _{y}\partial _{z}u^{\perp }(y,z)\,dz+\textstyle%
\sum\limits_{\pm }\partial _{z}u^{\perp }(y,\pm H(y))\partial _{y}H(y)=0\,.
\end{equation*}%
By the trace inequality we find that
\begin{multline*}
\left\vert \int_{\Upsilon (y)}\partial _{y}\partial _{z}u^{\perp
}(y,z)\,dz\right\vert ^{2}=\left( \textstyle\sum\limits_{\pm }\partial
_{z}u^{\perp }(y,\pm H(y))\right) ^{2}|\partial _{y}H(y)|^{2}\leq \\
\leq c|H(y)||\partial _{y}H(y)|^{2}\int_{\Upsilon (y)}|\partial
_{z}^{2}u^{\perp }(y,z)|^{2}\,dz\,.
\end{multline*}%
Thus, from the relation \eqref{log} which implies \eqref{logres}, we get
\begin{equation}
\left\vert K_{3}\right\vert \leq c\sup\limits_{y\in \lbrack \rho ,+\infty
)}\left\{ \left\vert G_{H^{3},\rho }(y)\right\vert ^{-1/2}\left\vert
\partial _{y}H(y)\right\vert \left\vert H(y)\right\vert ^{1/2}\right\}
J_{2}^{1/2}I_{1}^{1/2}\leq cJ_{2}^{1/2}I_{1}^{1/2}.  \label{K3est}
\end{equation}

We find by combining (\ref{K1byK3}) and (\ref{K3est}) that
\begin{equation*}
K_{1}\leq I_{2}+cJ_{2}^{1/2}I_{1}^{1/2}
\end{equation*}%
and it holds with (\ref{J3est}) respectively (\ref{J4est}) that
\begin{eqnarray}
\left\vert J_{3}\right\vert &\leq &c_{1}\left( \rho \right)
J_{1}^{1/2}I_{1}^{1/2}+c_{2}\left( \rho \right) J_{1}^{1/2}\left(
I_{2}+cJ_{2}^{1/2}I_{1}^{1/2}\right) ^{1/2},  \label{J3fin} \\
\left\vert J_{4}\right\vert &\leq &c_{1}\left( \rho \right)
J_{2}^{1/2}I_{1}^{1/2}+c_{2}\left( \rho \right) J_{2}^{1/2}\left(
I_{2}+cJ_{2}^{1/2}I_{1}^{1/2}\right) ^{1/2}.  \label{J4fin}
\end{eqnarray}%
Using first (\ref{star}) and (\ref{star3}), next (\ref{J3fin}) and (\ref%
{J4fin}) for $\rho $ large enough, and finally (\ref{I1est}), (\ref{J1est})
and (\ref{J2est}) we conclude that indeed
\begin{gather*}
\left\Vert \nabla _{x}^{2}u;L^{2}(\Pi _{\rho })\right\Vert
^{2}=I_{1}+4I_{2}+I_{3}\geq \\
\geq I_{1}+4I_{2}+J_{1}+J_{2}+2J_{3}+2J_{4}\geq \rule{0mm}{5mm} \\
\geq \frac{1}{2}\left( I_{1}+4I_{2}+J_{1}+J_{2}\right) \geq \frac{1}{2}%
\left( I_{1}+J_{1}+J_{2}\right) \geq \\
\geq c\min \{H^{-4}(\rho ),W_{H}(\rho ),W_{H^{3}}(\rho )\}\left\Vert
u;L^{2}(\Pi _{\rho })\right\Vert ^{2},
\end{gather*}
whenever $\rho $ is large enough.
\end{proof}

\begin{theorem}
\label{Theorem2}Suppose that $H$ satisfies Assumption \ref{ass2} and that 
\begin{equation*}
\lim\limits_{t\rightarrow \infty }\partial _{t}\left( \log \left( H\left(
t\right) \right) \right) =-\infty.
\end{equation*} 
Then the embedding $H^{2}(\Omega )$ in $L^{2}(\Omega )$ is compact and the 
spectrum of the problem \eqref{0.4} with the Neumann boundary conditions \eqref{0.6} on both 
sides of the peak is discrete.
\end{theorem}

Now the following statement, which, as mentioned in the beginning of the
paper, follows from a result in \cite{AdFu,Ev}.

\begin{corollary}
\label{Corollary2}If \eqref{0.2} holds, then the spectrum of the problem %
\eqref{0.4}, with the Neumann boundary conditions \eqref{0.6} on the sides of the peak, is
discrete.
\end{corollary}

\begin{remark}
Note that the functions $H(y)=y^{-\alpha }$ and $H(y)=\exp (-\alpha y)$, $%
\alpha >0$, do not satisfy the requirement in \eqref{0.2}. The functions $H(y)=\exp
(-y^{1+\alpha })$ with $\alpha >0$ however do.
\end{remark}

\section{The incomplete Dirichlet condition and a 2d-estimate}
Let the boundary conditions provide only the single stable condition
\begin{equation}
u(x)=0,\quad x\in \Sigma _{\rho }^{+}\quad (\mbox{ or }x\in \Sigma _{\rho
}^{-})\,.  \label{five}
\end{equation}%
Then Friedrich's inequality on section $\Upsilon (y)$ holds and,
consequently,
\begin{equation*}
\Vert \mathcal{A}\,\partial _{z}u;L^{2}(\Pi _{\rho })\Vert ^{2}\geq c\Vert
\mathcal{A}\,H^{-2}u;L^{2}(\Pi _{\rho })\Vert ^{2}
\end{equation*}%
for every positive weight function $y\mapsto \mathcal{A}(y)$.

The function $v=\partial _{z}u$ can be represented as the sum $%
v(x)=v_{0}(y)+v^{\perp }(x)$ where, for $y>\rho $, the component $v^{\perp }$
satisfies the first conditions in \eqref{ortog}. Therefore,
\begin{multline*}
\int_{\Pi _{\rho }}|\nabla _{x}^{2}u(x)|^{2}\,dx\geq \int_{\Pi _{\rho
}}|\nabla _{x}v(x)|^{2}\,dx\geq \int_{\rho }^{+\infty }2H(y)|\partial
_{y}v_{0}(y)|^{2}\,dy\ + \\
+\int_{\Pi _{\rho }}|\partial _{z}v^{\perp }(y,z)|^{2}\,dx+2\int_{\Pi _{\rho
}}\partial _{y}v_{0}(y)\partial _{y}v^{\perp
}(y,z)\,dx=:I_{4}+I_{5}+2I_{6}\,.
\end{multline*}%
Setting $Z_{H}(y)=H(y)^{-1}G_{H}(y)$, we get
\begin{equation*}
I_{4}\geq \int_{\rho }^{+\infty }2G_{H}(y)|v_{0}(y)|^{2}dy\geq \int_{\rho
}^{+\infty }2Z_{H}(y)H(y)|v_{0}(y)|^{2}dy=\Vert Z_{H}v_{0};L^{2}(\Pi _{\rho
})\Vert ^{2}.
\end{equation*}%
Friedrich's inequality implies $I_{5}=\Vert \partial _{z}v^{\perp
};L_{2}(\Pi _{\rho })\Vert \geq c\Vert H^{-2}v^{\perp };L^{2}(\Pi _{\rho
})\Vert ^{2}$. Furthermore,
\begin{multline*}
I_{6}=\int_{\Pi _{\rho }}\partial _{y}v_{0}(y)\partial _{y}v^{\perp
}(y,z)\,dx\leq I_{5}^{1/2}\bigg(\int_{\rho }^{+\infty }\frac{1}{2H(y)}\bigg|%
\int_{\Upsilon (y)}\partial _{y}v^{\perp }(y,z)\,dz\bigg|^{2}\bigg)^{1/2} \\
\leq I_{5}^{1/2}\bigg(\int_{\rho }^{+\infty }\frac{1}{2H(y)}\big|v^{\perp
}(y,H(y))\partial _{y}H(y)-v^{\perp }(y,-H(y))\partial _{y}H(y)\big|^{2}%
\bigg)^{1/2}.
\end{multline*}%
Thus $I_{6}\leq c|\partial _{y}H(\rho )|I_{5}^{1/2}I_{6}^{1/2}$ holds and
hence
\begin{equation*}
\left\Vert \nabla _{x}^{2}u;L^{2}(\Pi _{\rho })\right\Vert ^{2}\geq
c\left\Vert \min \{H^{-4};H^{-3}G_{H}\}u;L^{2}(\Pi _{\rho })\right\Vert
^{2}.\medskip
\end{equation*}

\begin{theorem}
\label{Theorem3}Suppose that $\lim\limits_{y\rightarrow \infty
}H(y)^{-3}G_{H}\left( y\right) =+\infty $. Then problem \eqref{0.4} with 
the boundary condition as in \eqref{five} has only
a discrete spectrum.
\end{theorem}

\begin{remark}The functions $H(y) = y^{-1 - \alpha}$ with $\alpha> 0$ meet the condition
in Theorem \ref{Theorem3}.
\end{remark}

\end{document}